\documentclass{amsart}

\usepackage{amsmath}
\usepackage{paralist}
\usepackage[usenames]{color}
\usepackage[colorlinks=true]{hyperref}
\usepackage[pagewise]{lineno}

\newtheorem{theorem}{Theorem}[section]
\newtheorem{lemma}[theorem]{Lemma}

\theoremstyle{definition}
\newtheorem{definition}[theorem]{Definition}

\theoremstyle{remark}
\newtheorem{remark}[theorem]{Remark}

\numberwithin{equation}{section}

\newcommand{\R}{\mathbb{R}}
\begin{document}

\title[Hyperbolicity of equilibria for a quasilinear nonlocal problem]{Stability and hyperbolicity of equilibria for a scalar nonlocal one-dimensional quasilinear parabolic problem}

\author{Alexandre N. Carvalho}
\address{Instituto de Ci\^{e}ncias Ma\-te\-m\'{a}\-ti\-cas e de Computa\c{c}\~{a}o\\ Universidade de S\~{a}o Paulo-Campus de S\~{a}o Carlos \\	Caixa Postal 668, S\~{a}o Carlos SP, Brazil
}
\email{andcarva@icmc.usp.br}
\thanks{The first author was supported by Grants FAPESP 2018/10997-6 and CNPq 306213/2019-2}

\author{Estefani M. Moreira}
\address{Instituto de Ci\^{e}ncias Ma\-te\-m\'{a}\-ti\-cas e de Computa\c{c}\~{a}o\\ Universidade de S\~{a}o Paulo-Campus de S\~{a}o Carlos \\	Caixa Postal 668, S\~{a}o Carlos SP, Brazil}

\curraddr{
}
\email{estefani@usp.br}
\thanks{The second author was supported by Grants FAPESP 2018/00065-9 and CAPES 7547361/D}



\begin{abstract}
In this work, we present results on stability and hyperbolicity of equilibria for a scalar nonlocal one-dimensional quasilinear parabolic problem. We show that this nonlocal version of the well-known Chafee-Infante equation bares some resemblance with the local version. However, its nonlocal characteristc requires a fine analysis of the spectrum of the associated linear operators, a lot more ellaborated than the local case. The saddle point property of equilibria is shown to hold for this quasilinear model.
\end{abstract}

\maketitle

\section{Introduction}

In this work, we consider the asymptotic behavior of solutions of the initial boundary value problem
\begin{equation}\label{eq_non-local}
\left\{\begin{split}
& u_t=a(\Vert u_x\Vert^2)u_{xx}+\lambda f(u),\ x \in (0,\pi), \  t>0,\\
& u(0,t)=u(\pi,0)=0, \ t\geq 0, \\
& u(\cdot,0)=u_0(\cdot)\in H^1_0(0,\pi),
\end{split}\right. 
\end{equation}
where $\lambda >0$ is a parameter, $a:\mathbb{R^+}\to [m,M]$, for $M>m>0$, is a globally Lipschitz continuous and non-decreasing function and $f \in C^2(\mathbb{R})$ is odd  and satisfies $f'(0)=1$,
\begin{equation} \label{eq_prop_f}
\limsup_{|u|\rightarrow +\infty } \frac{f(u)}{u} <0\
\hbox{ and }\ f''(u)u<0, \ u \neq 0.
\end{equation}

\bigskip

Before we proceed, let us make a simple remark about well-posedness of the problem \eqref{eq_non-local}. Consider the auxiliary initial boundary value problem
\begin{equation}\label{eq_nl_changed}
\left\{
\begin{aligned}
& w_\tau = w_{xx} + \frac{\lambda f(w)}{a(\| w_x\|^2)},  \, x\in (0, \pi),\, \tau>0,\\
& w(0, \tau)=w(\pi, \tau)=0,  \ \  \tau\geq 0,\\
& w(\cdot,0)=u_0(\cdot)\in H^1_0(0,\pi).
\end{aligned}
\right.
\end{equation}

Note that, the problem \eqref{eq_nl_changed} is locally well-posed and the solutions are jointly continuous with respect to time and initial conditions. Namely, the results of \cite[Section 6.7]{CLRBook} can be applied with the same proofs, even though the nonlinearity in \eqref{eq_nl_changed} is nonlocal in space.

Now, making the change of variables $t=\int_0^\tau a(\|w_x(\cdot,\theta)\|^2)^{-1}d\theta$ we have that $u(x,t)=w(x,\tau)$ is the unique solution of \eqref{eq_non-local}. This ensures that \eqref{eq_non-local} is globally well-posed (the same reasoning has been used in \cite{CLLM} to ensure well-posedness). If $u(\cdot ,u_0):\R^+\to H^1_0(0,\pi)$ denotes the solution of \eqref{eq_non-local}, we define the continuous map $T(t):H^1_0(0,\pi)\to H^1_0(0,\pi)$ by $T(t)u_0=u(t,u_0)$, $t\geq 0$.
The family $\{T(t):t\geq 0\}$ is a semigroup and it has a global attractor $\mathcal{A}$, that is, a compact invariant subset of $H^1_0(0,\pi)$ which attracts bounded subsets of $H^1_0(0,\pi)$ under the action of $\{T(t):t\geq 0\}$. The proof of that can also be achieved by employing the usual techniques of semilinear parabolic problems to \eqref{eq_nl_changed} and observing that the attractor of \eqref{eq_nl_changed} is also the attractor for \eqref{eq_non-local}.

A global solution of $\{T(t):t\geq 0\}$ is a continuous function $\xi: \R \mapsto H^1_0(0,\pi)$ such that $\xi(t+s)=T(t)\xi(s)$ for all $t\geq 0$ and for all $s\in \R$. It is well-known that (see, for example, \cite{Hale,CLRBook,BCLBook})
$$
\mathcal{A}=\{\xi(0): \xi:\R \to H^1_0(0,\pi) \hbox{ is a global bounded solution of } \{T(t):t\geq 0\} \}.
$$

Furthermore, this semigroup is gradient and the continuous functional\\ $V:H^1_0(0,\pi) \to \R$ given by
\begin{equation}\label{LyapunovFunction}
V(u) = \frac12\int_0^{\|u'\!\|^2}\!\!\! a(s) ds -\lambda\int_0^\pi \int_0^{u(x)}f(s)ds\, dx 
\end{equation}
is its Lyapunov functional. We also denote $\|u'\|$ by $\|u\|_{H^1_0(0,\pi)}$.

As a consequence, we have that $\{T(t):t\geq 0\}$ is a gradient semigroup (see, for example, \cite{Hale,CLRBook,BCLBook}) and if $\mathcal{E}$ denotes the set of equilibria of \eqref{eq_non-local}, that is, the set of solutions of
\begin{equation}\label{sl:non-local}
\left\{\begin{aligned}
& a(\Vert \varphi'\Vert^2)\varphi''+\lambda f(\varphi) =0,\ x \in (0,\pi) \nonumber\\
& \varphi(0)=\varphi(\pi)=0,
\end{aligned}\right. 
\end{equation} 
then $\mathcal{E}\subset \mathcal{A}$ and $\mathcal{A} = W^u(\mathcal{E})$, where 
\begin{equation*}
\begin{split}
W^u(\mathcal{E}) = \Big\{ u\in H^1_0(0,\pi): & \ \hbox{there exists a global solution } \xi:\R\to H^1_0(0,\pi)\\
& \hbox{ satisfying } \xi(0)=u \hbox{ and } \inf_{\varphi\in \mathcal{E}}\|\xi(t)-\varphi\|_{H^1_0(0,\pi)}\stackrel{t\to -\infty}{\longrightarrow} 0\Big\}.
\end{split}
\end{equation*}

We also know that 
$$
T(t)u_0 \stackrel{t\to \infty}{\longrightarrow} \mathcal{E}.
$$

When the set $\mathcal{E}$ has finite number $N$ of elements, that is $\mathcal{E}=\{\phi_1,\cdots, \phi_N\}$, more can be said of the asymptotics of \eqref{eq_non-local}. Namely, 
\begin{equation}\label{Att_charac}
\mathcal{A} = \bigcup_{n=1}^N W^u(\phi_n),
\end{equation}
for any global bounded solution $\xi:\R \to H^1_0(0,\pi)$ there are $n^-$ and $n^+$, $1\leq n^-,n^+\leq N$, such that 
$$
\phi_{n^-}\stackrel{t\to -\infty}{\longleftarrow}\xi(t)\stackrel{t\to +\infty}{\longrightarrow}\phi_{n^+}.
$$
and, for any $u_0\in H^1_0(0,\pi)$, there exists $n$, $1\leq n\leq N$, such that $T(t)u_0\stackrel{t\to \infty}{\longrightarrow} \phi_n$.

\bigskip

It has been proved in \cite{CLLM} that, assuming that $a$ is non-decreasing, the number of equilibria for \eqref{eq_non-local} is finite. In fact, the analysis done in \cite{CLLM} is for $f(u)=\lambda u - bu^3$, $b>0$, but it holds, without any change in the proofs, for the class of functions considered here.
%

\bigskip

Several aspects of the well-known theory for the corresponding semilinear problem (e.g. $a\equiv 1$) is yet to be unraveled in the quasilinear nonlocal case associated to a non-constant function $a$. For semilinear problems, with the aid of the variation of constants formula, one can prove important results like the saddle point property of hyperbolic equilibria, that is, the local stable and unstable manifolds are graphs tangent to the linear stable and unstable manifolds associated to the linearization around the equilibrium. The simple notion, hyperbolicity of equilibria of semilinear problems, becomes quite challenging in the quasilinear context. The Kirchhoff type diffusivity is yet a nonlocal term which tends to complicate the analysis of the model. In this paper we aim to consider the questions of stability and hyperbolicity of equilibria and start consolidating the notion of hyperbolicity for general systems for which a linearization process may not be available to interpret the local dynamics (at least not directly). 

\bigskip

Our analysis of the model \eqref{eq_non-local} is inspired by the local case $(a\equiv 1)$ for which a lot is known. Also, some analysis of these problems have been started from the point of asymptotics \cite{CH-VA-VC} exploiting the gradient structure of the model and from the point of view of finding solutions for the associated elliptic problem \cite{CH-R}. 

\bigskip

Models of this type arise, for example, in microwave heating of a thin ceramic cylinder (see \cite{Kri}), see also \cite{CH-VA-VC} for additional references on applications of the model. In the words of \cite{CH-VA-VC}, equilibria may appear in a much more complicated fashion than in the semilinear case, making the study of the asymptotic behavior of these problems interesting and challenging. 

\bigskip

Before we proceed, let us explain the meaning of hyperbolicity in the general context of a semigroup for which we may not be able to do a linearization process.

\bigskip Let $\phi$ be an equilibrium for \eqref{eq_non-local}, that is, a solution of \eqref{sl:non-local}. The notion of hyperbolicity of an equilibrium $\phi$ carries the idea that solutions starting near $\phi$ must converge to $\phi$ as $t\to +\infty$, or converge to $\phi$ as $t\to -\infty$ or may not stay near the equilibrium for an infinite interval of time.

\begin{definition}[Topological Hyperbolicity]\label{TH}
	We say that $\phi$ is topologically hyperbolic if $\{\phi\}$ is an isolated invariant set. In other words, there exists a $\delta>0$ for which any global solution $\xi:\R\to H^1_0(0,\pi)$, with $\sup_{t\in \R}\|\xi(t)-\phi\|_{H^1_0(0,\pi)}<\delta$, satisfies $\xi(t)=\phi$, for all $t\in \R$. 
\end{definition}

As a consequence of that (see, \cite{BCLBook}), any solution $\eta^\pm:J^\pm\to H^1_0(0,\pi)$, with $J^+=[t_0,\infty)$ or $J^-=(-\infty,t_0]$, such that $\|\eta^\pm(t)-\phi\|_{H^1_0(0,\pi)}<\delta $ for all $t\in J^\pm$, satisfies $\eta^\pm(t)\stackrel{t\to\pm\infty}{\longrightarrow}\phi$. 

\begin{definition}[Local Stable $W^s_{loc}(\phi)$ and Unstable Sets $W^u_{loc}(\phi)$] 
	
	Given a $\delta-$neigh\-borhood ${\mathcal O}_\delta(\phi)=\{u\in H^1_0(0,\pi): \|u-\phi\|_{H^1_0(0,\pi)}<\delta\}$ of $\phi$, the associated \emph{local stable and unstable sets} of $\phi$, respectively, are given by
	\begin{equation*}
	\begin{split}
	W^{s,\delta}_{loc}(\phi) = &\{u \in H^1_0(0,\pi): T(t) u \in {\mathcal O}_\delta \hbox{ for all } t \geq  0, \hbox{ and } T(t) u\stackrel{t\to\infty}{\longrightarrow} \phi\}, \\
	W^{u,\delta}_{loc}(\phi) = &\{u \in  H^1_0(0,\pi): \hbox{ there exists a global solution } \xi  \hbox{  of }  \{T(t)\colon t\geq 0\}  \\
	& \hspace{40pt} \hbox{ with }\xi(0)=u, \ \xi(t)\in {\mathcal O}_\delta \hbox{ for all }  t\leq 0 \hbox{ and }\xi(t)\stackrel{t\to -\infty}{\longrightarrow} \phi\}.
	\end{split}
	\end{equation*}
\end{definition}

When $\phi$ is topologically hyperbolic and $W^{u,\delta}_{loc}(\phi)=\{\phi\}$ we say that $\phi$ is asymptotically stable otherwise it is said unstable.

\begin{definition}[Strict Hyperbolicity]\label{SH} We say that $\phi$ is {\bf hyperbolic} if there are closed subspaces $X_u$ and $X_s$ of $H^1_0(0,\pi)$ with $H^1_0(0,\pi)=X_u\oplus X_s$ such that
	\begin{itemize}
		\item $\{\phi\}$ topologically hyperbolic.
		\item The local stable and unstable sets are given as graphs of Lipschitz functions $\theta_u:X_u\to X_s$ and 
		$\theta_s:X_s\to X_u$, with Lipschitz constants $L_s$, $L_u$ in $(0,1)$ and such that $\theta_u(0)=\theta_s(0)=0$ and, there exists $\delta_0>0$ such that, given $0<\delta<\delta_0,$ there are $0<\delta''<\delta'<\delta$ such that 
		\begin{equation*}
		\begin{split}
		\{\phi\!+\!(x_u,\theta_u(x_u))\!:\! x_u\!\in\! X_u,\ & \|x_u\|_{H^1_0(0,\pi)}\!<\!\delta''\}\!\subset\! W^{u,\delta'}_{loc}(\phi) \\
		&\!\subset\! \{\phi\!+\!(x_u,\theta_u(x_u))\!:\! x_u\!\in\! X_u, \|x_u\|_{H^1_0(0,\pi)}\!<\!\delta\}\\
		\{\phi\!+\!(\theta_s(x_s),x_s)\!:\! x_s\in X_s,\ & \|x_s\|_{H^1_0(0,\pi)}\!<\!\delta''\}\!\subset\! W^{s,\delta'}_{loc}(\phi)\\
		&\!\subset\! \{\phi\!+\!(\theta_s(x_s),x_s)\!:\! x_s\!\in\! X_s, \|x_s\|_{H^1_0(0,\pi)}\!<\!\delta\}.
		\end{split}
		\end{equation*}
	\end{itemize}
\end{definition}

\bigskip

Since $\{T(t):t\geq 0\}$ is gradient, for topological hyperbolicity is enough to ensure that $\phi$ is an isolated equilibrium (see \cite[Lemma 2.18]{BCLBook}). Strict hyperbolicity for semilinear problems is usually a consequence of dichotomy for the semigroup obtained by linearization around the equilibrium (the spectrum of the solution operator of the linearized equation is disjoint from the unit circle (see \cite[Section 4.1]{BCLBook}) and the saddle point property (see [Theorem 4.4]\cite{BCLBook}).

\bigskip 

Obtaining strict hyperbolicity for quasilinear problems is a task that has not been taken into account in the literature so far. In this paper, we consider the quasilinear, nonlocal, scalar, one dimensional parabolic problem \eqref{eq_non-local} and show that  we can still obtain strict hyperbolicity of equilibria. This is a particularly interesting problem for which one may change from finitely many equilibria to a finite number of clusters (containing possibly a continuous of equilibria) by changing only the assumptions on the function $a$.

\bigskip

Before we proceed, let us briefly recall a little of what is known for the case $a\equiv 1$. That corresponds to the well-known Chafee-Infante equation, given by 
\begin{equation}\label{eq:C-I}
\left\{ \begin{aligned}
& u_t=u_{xx}+\lambda f(u), \  x \in (0,\pi), \ t>0, \\
& u(0,t)=u(\pi,t)=0, \ t\geq 0, \\
& u(\cdot,0)=u_0(\cdot) \in H^1_0(0,\pi).
\end{aligned} \right. 
\end{equation}
In \cite{Chaf-Inf} and \cite{CH-IN}, the authors constructed a bifurcation sequence of the zero equilibrium for \eqref{eq:C-I}. There, the authors also present results on stability of the equilibria for \eqref{eq:C-I}, in the sense of \cite{HE}. The hyperbolicity of the nontrivial equilibria can also be proved for all values of the parameter using the results in \cite{Smoller}.

Suppose that $\varphi$ is an equilibrium of \eqref{eq:C-I}. Since \eqref{eq:C-I} is a semilinear problem, studying the stability and hyperbolicity of $\varphi$ is the same as localize the spectrum $\sigma(L)$ of the operator
$$
Lu=u''+\lambda f'(\varphi)u,
$$
which is called the linearization around $\varphi$. In \cite{HE}, we see that $\varphi$ is stable if all the eigenvalues of the linearization are non-positive and $\varphi$ is hyperbolic if none eigenvalue is in the imaginary axis. We can summarize the results on \eqref{eq:C-I} in the following

\begin{theorem}\label{theo:C-I}
	Consider $N=1,2,\cdots$. If $N^2< \lambda\leq (N+1)^2,$ then 
	\begin{itemize}
		\item[] \underline{Existence}: There are $2N+1$ equilibria of the equation \eqref{eq:C-I} \[\{0\}\cup\left\{\varphi_{j}^\pm: j=1,\dots, N\right\},\] where $\varphi^+_{j}$ and $\varphi^-_{j}$  have $j+1$ zeros in $[0,\pi]$ and $\varphi^-_{j}=-\varphi^+_{j}$.
		\item[] \underline{Stability}: When $ \lambda \leq 1$, $0$ is the only equilibrium of \eqref{eq:C-I} and it is stable. When $\lambda>1$, the positive equilibrium $\varphi_1^+$ and the negative equilibrium $\varphi_1^-$ are stable and any other equilibrium is unstable.
		\item[] \underline{Hyperbolicity}: For all $\lambda>0$, the equilibria are hyperbolic with the exception of $0$ in the cases $\lambda=N^2$,  for  $N=1,2,\cdots$.
	\end{itemize}
\end{theorem}

%
%
%
%
%
%

\bigskip

With respect to problem \eqref{eq_non-local}, the authors in \cite{CLLM} proved, using variational techniques and symmetry properties of solutions, the following bifurcation result
\begin{theorem}\label{theo:exist.equilibria}
	If $a(0)N^2< \lambda\leq a(0)(N+1)^2,$ then there are $2N+1$ equilibria of the equation \eqref{eq_non-local}; $\{0\}\cup\left\{\phi_{j}^\pm: j=1,\dots, N\right\}$,
	where $\phi^+_{j}$ and $\phi^-_{j}$  have $j+1$ zeros in $[0,\pi]$ and $\phi^-_{j}(x)=-\phi^+_{j}(x)$ for all $x\in[0,\pi]$ and $\phi^+_{j}(x)>0$ for all $x\in(0,\frac{\pi}{j})$.
\end{theorem}

%

In this paper, our main result is the following 

\begin{theorem}\label{theo:stab-hyp}
	The sequence of bifurcation given in Theorem \ref{theo:exist.equilibria} satisfies:
	\begin{itemize}
		\item[] {\underline{Stability}}: If $\lambda\leq a(0)$, $0$ is the only equilibrium of \eqref{eq:C-I} and it is stable. If $\lambda>a(0)$, the positive equilibrium $\phi_1^+$ and the negative equilibrium $\phi_1^-$ are stable and any other equilibrium is unstable.

		\item[] {\underline{Hyperbolicity}}:  For all $\lambda>0$, the equilibria are hyperbolic with the exception of $0$ in the cases $\lambda=a(0)N^2$,  for  $N=1,2,\cdots$.

	\end{itemize}
\end{theorem}

\bigskip

Observe that, Theorem \ref{theo:exist.equilibria}  assures that we have the existence of a sequence of bifurcation of equilibria for \eqref{eq_non-local}. Here, our aim is to obtain information about the asymptotic behavior and try to formulate a result for the equilibria of \eqref{eq_non-local} such as Theorem \ref{theo:C-I}.

Since \eqref{eq_non-local} is a quasilinear nonlocal problem, linearization procedures leading to interesting information about the original systems are, in general, not available (for the local case, see \cite{Phillipo}). Hence the techniques used in Theorem \ref{theo:C-I} will not apply.

Notice that \eqref{eq_non-local} is a quasilinear problem, so we cannot immediately apply the techniques used in \eqref{eq:C-I}. To circumvent this inconvenience we will work with \eqref{eq_nl_changed} instead and then re-interpret the result for \eqref{eq_non-local}. We remark that the time change of variables involved will be different for each solution and one must be careful when transferring results between the two formulations.

\bigskip

Clearly, if $\phi$ is an equilibrium of \eqref{eq_non-local}, then $\phi$ is also an equilibrium of \eqref{eq_nl_changed}.

\bigskip

Therefore, for the equation \eqref{eq_nl_changed}, we can talk about linearization around some equilibrium and make a spectral analysis of the resulting linear operator. If $\phi$ is an equilibrium of $\eqref{eq_nl_changed}$, then this linearization procedure results in the operator $L: D(L)\subset H^1_0(0,\pi)\to H^1_0(0,\pi)$, with $D(L)=H^2(0,\pi)\cap H^1_0(0,\pi)$ and 
\begin{equation}\label{eq:linearization}
L v=v'' +\lambda\frac{f'(\phi)}{a(\|\phi'\|^2)}v-\tfrac{2\lambda^2 a'(\|\phi'\|^2)}{a(\|\phi'\|^2)^3}f(\phi)\int_{0}^{\pi} f(\phi(s))v(s)ds.
\end{equation}

Observe that the linearization around $\phi$ is a self-adjoint operator with compact resolvent having a nonlocal term. Because of that, its spectrum consists only of eigenvalues, which are not necessarily simple.

%
%
%

\bigskip

We start Section 2 showing that there exists a linear operator $L$ whose spectrum carries the information about the stability, instability and hyperbolicity of equilibria for \eqref{eq_nl_changed}. Then we present some existing results for nonlocal linear operators and also some facts about eigenvalues and eigenvectors of \eqref{eq:linearization}. In section 3, we develop the results on stability of equilibria of \eqref{eq_non-local}. In Section 4, we finally conclude the proof of Theorem \ref{theo:stab-hyp} by showing {the} hyperbolicity of the equilibria of \eqref{eq_non-local}. In Section 5, we discuss the results in this paper.

\section{The linearized operator and some spectral properties of associated nonlocal operators}\label{S2}

Again, the linearization of \eqref{eq_nl_changed} around $\phi$ is given by
\begin{equation}\label{Lepsphi}
L_{\varepsilon}^\phi v=v'' +\frac{\lambda f'(\phi)}{a(\|\phi'\|^2)}v+\varepsilon  f(\phi)\int_{0}^{\pi} f(\phi(s))v(s)ds,
\end{equation}
in the particular case $\varepsilon=-\tfrac{\lambda^22a'(\|\phi'\|^2)}{a(\|\phi'\|^2)^3} <0$.

In fact, 
we have 
\begin{equation*}
\frac{f(u+h)}{a(\|u'+h'\|^2)}\!-\!\frac{f(u)}{a(\|u'\|^2)}\!=\!\frac{f(u+h)-f(u)}{a(\|u'+h'\|^2)}\!+\!\left(\!\frac{a(\|u'\|^2)-a(\|u'+h'\|^2)}{a(\|u'+h'\|^2)a(\|u'\|^2)}\!\right)\!f(u).
\end{equation*}
Now, for $u,h\in H^1_0(0,\pi)$, and $\sigma_i:[0,\pi]\to (0,1)$, $i=1,2$, \[f(u+h)-f(u)= f'({u+\sigma_1 h})h\] and 
$$
a(\|u'\|^2)-a(\|u'+h'\|^2)\!=\!-a'({ (1-\sigma_2)\|u'\|^2\!+\! \sigma_2 \|u'+h'\|^2})\left(\|u'\!+\!h'\|^2\!-\!\|u'\|^2\right).
$$

Define $D \in L(H^1_0(0,\pi), L^2(0,\pi))$ by
\begin{equation*}
Dh=\frac{f'(u)h}{a(\|u'\|^2)}-\frac{2a'(\|u'\|^2)}{a(\|u'\|^2)^2}\left(\int_{0}^{\pi} u'h'\right)f(u).
\end{equation*}

It follows, from the above, that
$$
\frac{1}{\|h'\|}\left\|\frac{f(u+h)}{a(\|u'+h'\|^2)}-\frac{f(u)}{a(\|u'\|^2)} - D h\right\| \stackrel{\|h'\| \rightarrow 0}{\longrightarrow }0.
$$

Therefore, the linearization of \eqref{eq_nl_changed} around $\phi$ is given by the equation
\begin{equation*}
v_t = Lv
\end{equation*}
where $D(L)=H^2(0,\pi)\cap H^1_0(0,\pi)$ and 
$$
Lv=v'' +\frac{\lambda f'(\phi)}{a(\|\phi'\|^2)}v+\frac{2\lambda a'(\|\phi'\|^2)}{a(\|\phi'\|^2)^2}\left(\int_{0}^{\pi} \phi''v\right)f(\phi), \ v\in D(L).
$$
Since $\phi$ is an equilibrium, we conclude that $L=L_{\varepsilon}^\phi$.

\bigskip

We recall that, from the above, we may infer that the stability, instability and hyperbolicity of the equilibrium solution $\phi$ of \eqref{eq_nl_changed} from the analysis of the spectrum of $L_\varepsilon^\phi $. In fact, if all the eigenvalues of $L_\varepsilon^\phi$ are negative, $\phi$ is exponential stability, if at least one eigenvalue of $L_\varepsilon^\phi$ is positive, $\phi$ will be unstable and, if $0$ does not belong to the spectrum of $L_\varepsilon^\phi$, $\phi$ is hyperbolic (see \cite[Sections 5.1 and 5.2]{HE}).

\bigskip

Since the change of variables that allows us to relate solutions of \eqref{eq_nl_changed} with solutions of \eqref{eq_non-local} does not interfere with the state variable, the local unstable and stable sets for \eqref{eq_nl_changed} and \eqref{eq_non-local} will also have the same invariance properties. Also, if a solution converges to an equilibrium of \eqref{eq_nl_changed} its corresponding solution of \eqref{eq_non-local} will also converge to the same equilibrium (as $t$ tends to plus or minus infinity). The rate of convergence is scaled by the change of variables (bounds on $a$). Hence, the unstable and stable sets remain the same. In particular, stability and instability properties are the same. For hyperbolicity, if zero is not an eigenvalue of $L$, the linear stable and unstable manifolds associated to $L$ are the spaces $X_s$ and $X_u$ of Definition \ref{SH} (we will return to this in Sections 3 and 4).

Having established that, the knowledge of the spectrum will provide full information about the stability, instability and hyperbolicity of equilibria we set out to analyze it.

\bigskip

For $\varepsilon \in \mathbb{R}$, define the operator $L_\varepsilon: H^2(0,\pi)\cap H^1_0(0,\pi)\subset L^2(0,\pi)\to L^2(0,\pi) $
\begin{equation}\label{eq:nl-lineariz}
L_\varepsilon u(x) = u'' + p(x)u +\varepsilon c(x)\int_{0}^{\pi} c(s)u(s)ds.
\end{equation}
where $p,c:[0,\pi]\to \R$ are continuous functions with $c\not\equiv 0$.

When $\varepsilon=0$, the operator $L_0u=u''+p(x)u$ is a 
Sturm-Liouville operator. Hence, $L_0$ is a self-adjoint with compact resolvent and its spectrum consists of a decreasing sequence of simple eigenvalues, that is,
$$
\sigma(L_0)=\{\gamma_j: j=1,2,3\cdots \} 
$$
with, $\gamma_j > \gamma_{j+1}$ and $\gamma_j \longrightarrow -\infty$ as  $j \rightarrow +\infty$.

\bigskip 

Note that, for all $\varepsilon \in \mathbb{R}$, we can decompose $L_\varepsilon$ as sum of two operators 
$$
L_\varepsilon u=L_0 u + \varepsilon Bu
$$
where $Bu=c(x)\int_0^\pi c(s)u(s)ds$, for all $u \in H^2(0,\pi)\cap H^1_0(0,\pi)$, is a bounded operator with rank one. It is easy to see that $L_\varepsilon$ is also self-adjoint with compact resolvent. 

\bigskip

Then, we write $\{ \mu_j(\varepsilon): j=1,2,3,\cdots\}$ to represent the eigenvalues of $L_\varepsilon$, ordered in such a way that, for $j=1,2,3,\cdots$, the function $\mathbb{R}\ni \varepsilon \mapsto \mu_j(\varepsilon)\in \R$ satisfies $\mu_j(0)=\gamma_j$.

\bigskip

Throughout this paper we will use, in an essential way, the Theorems 3.4 and 4.5 of \cite{DavidsonDodds}. We summarize these results in the following
\begin{theorem}\label{DD}
Let $L_\varepsilon$ and $\{ \mu_j(\varepsilon): j=1,2,3,\cdots \}$ be as defined before. The following holds:
	\begin{itemize}
	 \item[i)] For all $j=1,2,3,\cdots $, the function  $\mathbb{R}\ni \varepsilon \mapsto \mu_j(\varepsilon)\in \R$ is non-decreasing.
	 \item[ii)] If for some $j=1,2,3\cdots $ and $\varepsilon \in \mathbb{R}$, $\mu_j(\varepsilon)\notin \{\gamma_k:k =1,2,3,\cdots \}$, then $\mu_j(\varepsilon)$ is a simple eigenvalue of $L_\varepsilon$. 
	\end{itemize}
\end{theorem}

So far as we can tell, there are several works exploring properties of operators such as $L_\varepsilon$ (see, for example,\cite{Freitas,Catchpole,DavidsonDodds,Dodds}). 


%


\bigskip



Suppose that $a(0)N^2< \lambda \leq a(0)(N+1)^2,$ for $N=1,2,\cdots$. In that case, Theorem \ref{theo:exist.equilibria} ensures that the set of equilibria is given by
$$
\left\{ \phi_j^\pm : 0\leq j \leq N \right\} \mbox{ where } \phi_j^- = - \phi_j^+ \mbox{ and } \phi_0^\pm =0.
$$

According to Theorem \ref{theo:exist.equilibria}, if $\phi$ is an equilibrium of \eqref{eq_nl_changed} then $\psi=-\phi$ is also an equilibrium. Since $f$ is odd and $f'$ is an even function, we have $f(\psi)=-f(\phi)$ and $f'(\psi)=f'(\phi)$. Hence, the linearization around $\phi$ and around $\psi$ are the same and we can restrict our analysis to one of them. 

As a consequence of that, for $j=1,\cdots, N$, we only need to study the stability and hyperbolicity of $\phi_j^+$ and the same result will hold for $\phi_j^-$, automatically. For this reason, in what follows, we write $\phi_j$ to denote $\phi_j^+$ and restrict our analysis to this case.

We start with some facts about the local part of $L^\phi_{0}$. 
Note that, if $\phi$ is an equilibrium for \eqref{eq_nl_changed},
$$
\frac{f'(\phi(\pi-x))}{a(\|\phi'\|^2)}=\frac{f'(\pm\phi(x))}{a(\|\phi'\|^2)}=\frac{f'(\phi(x))}{a(\|\phi'\|^2)},\ x \in [0,\pi],
$$
where we have used the following facts:
\begin{itemize}
	\item[i)] $\phi_{j}(\pi -x)=(-1)^{j-1}\phi_{j}(x),\  x \in [0,\pi], \ j=1,\dots, N, $ and 
	\item[ii)] $f'(-u)=f'(u)$ for all $u \in \mathbb{R}$.
\end{itemize}

\section{Stability and instability of the equilibria}

Let $\phi\neq 0$ with $\phi'(0)>0$ be an equilibrium solution of \eqref{eq_non-local} and consider the semilinear problem 
\begin{equation}\label{sl_non-local}
\left\{\begin{split}
& u_t=\bar{a}u_{xx}+\lambda f(u),\ x \in (0,\pi), \  t>0,\\
& u(0,t)=u(\pi,0)=0, \ t\geq 0, \\
& u(\cdot,0)=u_0(\cdot)\in H^1_0(0,\pi),
\end{split}\right. 
\end{equation}
with $\bar{a}=a(\Vert \phi'\Vert^2)$. Note that $\phi$ is also an equilibrium of \eqref{sl_non-local} and the linearization of \eqref{sl_non-local} around $\phi$ is given by
\begin{equation}\label{l_non-local}
\left\{\begin{split}
& w_t=\bar{a}w_{xx}+\lambda  f'(\phi)w,\ x \in (0,\pi), \  t>0,\\
& w(0,t)=w(\pi,0)=0, \ t\geq 0, \\
& w(\cdot,0)=w_0(\cdot)\in H^1_0(0,\pi),
\end{split}\right. 
\end{equation}

The following result synthesizes the spectral properties of the linearized operator

\begin{lemma}\label{Lemma-Henry}
The spectrum of the operator
\begin{equation}
\label{Op_A}
\left\{
\begin{split}
& L_0^\phi:D(L_0^\phi)\subset L^2(0,\pi)\to L^2(0,\pi), \\
& D(L_0^\phi)=H^2(0,\pi)\cap H^1_0(0,\pi), \\
& L_0^\phi v=v''+ \frac{\lambda f'(\phi)}{a(\|\phi'\|^2)}v,\  v\in D(L_0^\phi).
\end{split}
\right.
\end{equation}
satisfies
\begin{itemize}
\item[i)] If $\phi(x)> 0$ in $(0,\pi)$ then $L_0^\phi$ has only negative eigenvalues.
\item[ii)] If $\phi(x)= 0$ for some $x \in (0,\pi)$ then $L_0^\phi$ has at least one positive eigenvalue.
\item[iii)] $0$ is always in the resolvent of $L_0^\phi$.
\end{itemize}
\end{lemma}

The parts $i)$ e $ii)$ were proved by \cite{Chaf-Inf} and can also be found in \cite[Section 5.3]{HE}. Part $iii)$ is a consequence of the results in \cite[Section F of Chapter 24]{Smoller}.

Our goal is to show that, the conclusion of Lemma \ref{Lemma-Henry} remains the same when $L_0^\phi$ is replaced by $L_\varepsilon^\phi$, with $\varepsilon=-\tfrac{\lambda^2 2a'(\|\phi'\|^2)}{a(\|\phi'\|^2)^3}$.

\subsection{Stability and hyperbolicity of the equilibria $0$ and $\phi_1$}

$\quad $

\smallskip

If $\phi =0$, since $f(0)=0$, $f'(0)=1$, the operator $L_{\varepsilon_0}^0$ reduces to
$$
L_{\varepsilon_0}^0 u=u''+\frac{\lambda}{a(0)}u.
$$
Its spectrum is given by $\left\{-n^2+\tfrac{\lambda}{a(0)}:\ n=1,2,\cdots \right\}$. In particular, we have that when $\lambda< a(0), $ $0$ is exponentially stable, and it is the only equilibrium solution (with the assumption that $a$ is non-decreasing) of \eqref{eq_nl_changed}. Hence, $0$ is globally exponentially stable.

Thus, there is a neighborhood $V$ of $0$ and constants $K,\beta>0$ in such a way that, for all solutions $w$ of \eqref{eq_nl_changed} with $w(\cdot, 0)=u_0 \in V$ we have
$$
\|w(\tau)-0\|_{H^1_0(0,\pi)}\leq Ke^{-\beta \tau}\|u_0\|_{H^1_0(0,\pi)}, \quad \forall \tau \geq 0.
$$

Now, recall that, given a solution $u$ of \eqref{eq_non-local} we have $u(\cdot,0)=u_0 \in V$, 
$u(t)=w(\tau)$ where $t= \int_0^\tau a(\|w_x(\cdot,\theta)\|^2)^{-1}d\theta$, hence $\tfrac{\tau}{M}\leq t\leq \tfrac{\tau}{m}$.

 Therefore, for all $t \geq 0$,
$$
\|u(t)\|_{H^1_0(0,\pi)}\leq \|w(\tau)-0\|_{H^1_0(0,\pi)}\leq Ke^{-{\beta}\tau }\|u_0\|_{H^1_0(0,\pi)} \leq Ke^{-{\beta} mt}\|u_0\|_{H^1_0(0,\pi)}.
$$

Therefore, $0$ is an equilibrium exponentially stable of \eqref{eq_non-local}. 

\begin{remark}
Observe how interesting the previous result is. The problem \eqref{eq_non-local} is a quasilinear problem for which we are not able to talk about linearization around an equilibrium. Despite that, we can obtain a result on stability using the  related semilinear problem \eqref{eq_nl_changed}. The exponential rate of attraction is different for each solution starting in $V$ but it has a common bound in the case considered here.
\end{remark}
\bigskip

Suppose that $\lambda > a(0)$. Clearly, the equilibrium $0$ is unstable in this situation, that is, there exists a $\delta_0>0$ and constants $K,\beta>0$ in such a way that, for each $\delta<\delta_0$, a $0<\delta'<\delta$ and a global solution of \eqref{eq_nl_changed} $\eta:\R \to H^1_0(0,\pi)$ such that
\begin{equation*}
\begin{split}
&u_0\in H^1_0(0,\pi),\  \|u_0\|_{H^1_0(0,\pi)}<\delta',\\
&\eta(0)=u_0\ \hbox{and}\ \|\eta(\tau)\|_{H^1_0(0,\pi)}\leq \delta, \hbox{ for all } \tau\leq 0, \\
&\|\eta(\tau)-0\|_{H^1_0(0,\pi)}\leq K e^{\beta \tau}\|u_0 - 0\|_{H^1_0(0,\pi)}, \hbox{ for all } \tau \leq 0.
\end{split}
\end{equation*}
as before, making the change in the time-variable $t=\int_0^\tau a(\|\eta_x(\cdot,\theta)\|^2)^{-1}d\theta$, $\xi(t)=\eta(\tau)$ is a global solution of \eqref{eq_non-local} and, in this case, $\tfrac{\tau}{m} \leq t \leq \tfrac{\tau}{M}$. Therefore, for all $t \leq 0$,
$$
\|\xi(t)\|_{H^1_0(0,\pi)}\leq \|\eta(\tau)-0\|_{H^1_0(0,\pi)}\leq Ke^{{\beta}\tau}\|u_0\|_{H^1_0(0,\pi)}\leq Ke^{{\beta} mt}\|u_0\|_{H^1_0(0,\pi)}.
$$

Also, from Theorem \ref{theo:exist.equilibria}, there exists a positive equilibrium $\phi_1^+$ (and also $\phi_1^-=-\phi_1^+$) for \eqref{eq_nl_changed}, that is, $\phi_1^+(x)>0$ em $(0,\pi)$ and $\phi_1^+(0)=\phi_1^+(\pi)=0$. As remarked before, the linearization around both equilibria $(\phi_1^\pm)$ are the same, so we consider only the case $\phi_1^+$ and denote both by $\phi_1$.

\begin{theorem}
Assume that $a$ is also a non-decreasing function. Then $\phi_1$ exponentially stable and, consequently, hyperbolic.
\end{theorem} 
\begin{proof}
Consider the operators $L^{\phi_1}_{\varepsilon_1}$ and $L^{\phi_1}_{0}$ given by \eqref{eq:linearization} and \eqref{Op_A} with $\phi$ replaced by $\phi_1$. Denote the spectrum of $L^{\phi_1}_{\varepsilon}$ by $\sigma(L^{\phi_1}_{\varepsilon})=\{\mu^1_n(\varepsilon) : n=1,2,3,\cdots \}$.
Observe that $L^{\phi_1}_{\varepsilon_1}$ can be written as 
$$
L^{\phi_1}_{\varepsilon_1}u=L^{\phi_1}_{0}u + \varepsilon_1 f(\phi_1)\int_{0}^{\pi} f(\phi_1)u,
$$
where $\varepsilon_1=- \frac{2\lambda^2  a'(\|\phi_1'\|^2)}{a(\|\phi_1'\|^2)^3}$. Since $a$ is non-decreasing $\varepsilon_1\leq 0$.

Observe that $L^{\phi_1}_{0}$ has only negative eigenvalues since $\phi_1$ is an equilibrium of the equation \begin{equation}\label{eq:C-I1}
u_t= u_{xx}+\frac{f(u)}{a(\|(\phi_1)_x\|^2)}.
\end{equation} 
The linearization around $\phi_1$ of \eqref{eq:C-I1} is given by 
$$
u_t=L^{\phi_1}_{0}u.
$$
Recall that $\phi_1$ is the positive equilibrium of the Chafee-Infante equation \eqref{eq:C-I1} and its stability is well-known (see \cite{Chaf-Inf}). Therefore, we have that \[\cdots < \mu^1_{n+1}(0) < \mu^1_{n}(0) < \cdots < \mu^1_2(0) <\mu^1_{1}(0) < 0.\]

Then, applying Theorem \ref{DD}, the eigenvalues are non-decreasing functions of $\varepsilon$. Since, $\varepsilon_1\leq 0$, we conclude that
$$
\mu^1_n(\varepsilon_1) \leq \mu^1_1(0)<0, \quad \forall n=1,2,3,\cdots
$$
and follows the stability and hyperbolicity of $\phi_1$ for \eqref{eq_nl_changed}. Now, we need to transfer this information to \eqref{eq_non-local}.	

We have that the equilibrium $\phi_1$ \eqref{eq_nl_changed} is exponentially stable by linearization (see \cite[Section 5.1]{HE}). Therefore, there is a neighborhood $V$ of $\phi_1$, and positive constants $K$ and $\beta$ such that, for each $u_0 \in V$, the solution $w$ of \eqref{eq_nl_changed} through $u_0$ satisfies
$$
\| w(\tau)-\phi_1\|_{H^1_0(0,\pi)}\leq K e^{-\beta \tau}\|u_0 - \phi_1\|_{H^1_0(0,\pi)}, \quad \forall \tau \geq0.
$$
Therefore, the solution of \eqref{eq_non-local} is given by $u(t)=w(\tau)$, $t=\int_0^\tau a(\|w_x(\cdot,\theta)\|^2)^{-1}d\theta$, where $w$ is the solution of \eqref{eq_nl_changed} with the same initial condition. Hence, for $u_0 \in V$,
\begin{equation*}
\begin{split}
\| u(t)-\phi_1\|_{H^1_0(0,\pi)}=\| w(\tau)-\phi_1\|_{H^1_0(0,\pi)}&\leq K e^{-\beta\tau}\|u_0 - \phi_1\|_{H^1_0(0,\pi)} \\
&\leq K e^{-{\beta }mt}\|u_0 - \phi_1\|_{H^1_0(0,\pi)},\quad \forall t\geq 0.
\end{split}
\end{equation*}
\end{proof}

\subsection{The instability of the sign changing equilibria}
Let $a(0)N^2<\lambda \leq a(0)(N+1)^2$, $N =2,3,4,\cdots$. 

For $\phi_j=\phi_j^+,$  $j\in\{2,\dots, N\}$, consider the operators $L^{\phi_j}_{\varepsilon_j}$ and $L^{\phi_j}_{0}$ given by \eqref{eq:linearization} and \eqref{Op_A} with $\phi$ replaced by $\phi_j$. Hence,
\[
\begin{split} 
L^{\phi_j}_{\varepsilon_j} u
&=u'' +\frac{f'(\phi_j)}{a(\|\phi_j'\|^2)}u+\varepsilon_j f(\phi_j)\int_{0}^{\pi} f(\phi_j(s))u(s)ds
\end{split}
\]
with $\varepsilon_j=-\frac{2\lambda^2 a'(\|\phi_j'\|^2)}{a(\|\phi_j'\|^2)^3}$. Denote by $\sigma(L^{\phi_j}_{0})=\{\gamma_k^j : k=1,2,3\cdots \}$ the spectrum of $L^{\phi_j}_{0}$. Then, $\gamma_{k+1}^j<\gamma_{k}^j$, for all $k=1,2,3,\cdots$, $\gamma_k^j$ is a simple eigenvalue and we denote by $u_k$ its associated eigenfunction that satisfies $u_k'(0)=1$. We know that $u_k$ has $k+1$ zeros in $[0,\pi].$

Since $L^{\phi_j}_{0}$ is a Sturm-Liouville operator, its eigenvalues are all simple. Hence, for all $x\in [0,\pi],$
$$
u_k(\pi -x)=\left\{
\begin{split}
&u_k(x),\mbox{ if } k \mbox{ is odd}, \\
&-u_k(x),\mbox{ if } k \mbox{ is even}.
\end{split}\right.
$$
Consequently, for all $x \in [0,\frac{\pi}{2}],$
$$
u_k(x+\tfrac{\pi}{2})=\left\{
\begin{split}
&u_k(\tfrac{\pi}{2}-x),\mbox{ if } k \mbox{ is odd}, \\
&-u_k(\tfrac{\pi}{2}-x),\mbox{ if } k \mbox{ is even}.
\end{split}\right.
$$

\begin{lemma}\label{theo:eigenvalues}The following holds
	\begin{itemize}
		\item[i)] If $j$ is even, $\gamma_{2k-1}^j$ is an eigenvalue of $L^{\phi_j}_{\varepsilon_j}$ and $u_{2k-1}$ is an associated eigenfunction, $k =1,2,3,\cdots$.
		\item[ii)] If $j$ is odd, $\gamma_{2k}^j$ is an eigenvalue of $L^{\phi_j}_{\varepsilon_j}$ and $u_{2k}$ is an associated eigenfunction, $k=1,2,3,\cdots$.
	\end{itemize}
\end{lemma}
\begin{proof} In both cases we prove that $f(\phi_j)$ is orthogonal to $u_k$ and that will imply the desired results.

\begin{itemize}
\item[i)] Recall that $f$ is an odd function and since $j$ is even, we have that $\phi_j(x+\frac{\pi}{2})=-\phi_j(\frac{\pi}{2}-x),$ for $x \in [0,\frac{\pi}{2}].$ Then, for $k$ odd we have that
$$
\begin{aligned}
\int_{0}^{\pi} f(\phi_j(s))u_k(s)ds&=\int_{0}^{\frac{\pi}{2}} f(\phi_j(s))u_k(s)ds+\int_{\frac{\pi}{2}}^{\pi} f(\phi_j(s))u_k(s)ds \\
&\int_{0}^{\frac{\pi}{2}} f(\phi_j(s))u_k(s)ds+\int_0^{\frac{\pi}{2}} f(\phi_j(s+\tfrac{\pi}{2}))u_k(s+\tfrac{\pi}{2})ds \\
&\int_{0}^{\frac{\pi}{2}} f(\phi_j(s))u_k(s)ds+\int_0^{\frac{\pi}{2}} f(-\phi_j(\tfrac{\pi}{2}-s))u_k(\tfrac{\pi}{2}-s)ds \\
&\int_{0}^{\frac{\pi}{2}} f(\phi_j(s))u_k(s)ds-\int_0^{\frac{\pi}{2}} f(\phi_j(\tfrac{\pi}{2}-s))u_k(\tfrac{\pi}{2}-s)ds \\
&\int_{0}^{\frac{\pi}{2}} f(\phi_j(s))u_k(s)ds-\int_0^{\frac{\pi}{2}} f(\phi_j(s))u_k(s)ds = 0
\end{aligned}
$$		
Hence
$$
L^{\phi_j}_{\varepsilon_j}u_k=L^{\phi_j}_{0}u_k-\tfrac{2\lambda^2 a'(\|\phi_j'\|^2)}{a(\|\phi_j'\|^2)^3}f(\phi_j)\int_{0}^{\pi} f(\phi_j)u_k=\gamma_k^j u_k,
$$	
proving the result.

\item[ii)] If $j$ is odd, we have that $\phi_j(x+\frac{\pi}{2})=\phi_j(\frac{\pi}{2}-x),$ for $x \in [0,\frac{\pi}{2}].$ Then, for $k$ even we have that
$$
\begin{aligned}
\int_{0}^{\pi} f(\phi_j(s))u_k(s)ds&=\int_{0}^{\frac{\pi}{2}} f(\phi_j(s))u_k(s)ds+\int_{\frac{\pi}{2}}^{\pi} f(\phi_j(s))u_k(s)ds \\
&\int_{0}^{\frac{\pi}{2}} f(\phi_j(s))u_k(s)ds+\int_0^{\frac{\pi}{2}} f(\phi_j(s+\tfrac{\pi}{2}))u_k(s+\tfrac{\pi}{2})ds \\
&\int_{0}^{\frac{\pi}{2}} f(\phi_j(s))u_k(s)ds-\int_0^{\frac{\pi}{2}} f(\phi_j(\tfrac{\pi}{2}-s))u_k(\tfrac{\pi}{2}-s)ds \\
&\int_{0}^{\frac{\pi}{2}} f(\phi_j(s))u_k(s)ds-\int_0^{\frac{\pi}{2}} f(\phi_j(s))u_k(s)ds = 0
\end{aligned}
$$
\end{itemize}
	
Hence
$$
L^{\phi_j}_{\varepsilon_j}u_k=L^{\phi_j}_{0}u_k-\tfrac{2\lambda^2 a'(\|\phi_j'\|^2)}{a(\|\phi_j'\|^2)^3}f(\phi_j)\int_{0}^{\pi} f(\phi_j(s))u_k(s)ds=\gamma_k^j u_k,
$$
proving the result.
\end{proof}

Observe that the previous result helps us to identify part of the spectrum of $L^{\phi_j}_{\varepsilon_j}$. We can now prove the following

\begin{theorem} The equilibria $\phi_j$ are unstable, for $j\geq 2$ even.	
\end{theorem}
\begin{proof}

In this case, we can find $x_0 \in (0,\pi)$ such that $\phi_j(x_0)<0$ and $\phi_j'(x_0)=0$. Observe that $\phi_j''(x_0)=-\frac{\lambda f(\phi_j(x_0))}{a(\|\phi_j'\|^2)}>0$.

\bigskip

Consider $w$ the solution of
\begin{equation*}
\left\{\begin{aligned}
& 
w''+\frac{\lambda f'(\phi_j)}{a(\|\phi_j'\|^2)}w =0,\ x \in (0,\pi),\\
& w(0)=0, \ w'(0)=1.
\end{aligned}\right.
\end{equation*}

Observe that $v=\phi_j'$ is also a solution of the above problem.  In fact, 
$$
v''=(\phi_j'')' = \frac{-\lambda}{a(\|\phi_j'\|^2)}(f(\phi_j))'=\frac{-\lambda}{a(\|\phi_j'\|^2)}f'(\phi_j)\phi_j'=\frac{-\lambda}{a(\|\phi_j'\|^2)}f'(\phi_j)v
$$
Hence, their Wronskian $W(w,v)$ is constant and is given by  
$$
W(w,v)=w'(x)v(x)-w(x)v'(x)=\phi_j'(0)>0
$$ 
and for $x=x_0$ we have
$$
w'(x_0)v(x_0)-w(x_0)v'(x_0)=w'(x_0)\phi_j'(x_0)-w(x_0)\phi_j''(x_0)>0\Rightarrow w(x_0)< 0.
$$

From the above considerations and using the results in \cite[Page 122]{HE}, the first eigenvalue $\gamma_1^j$ of $L^{\phi_j}_{0}$ is positive.
From the previous theorem, we have that $\gamma_1^j$ is also an eigenvalue of $L^{\phi_j}_{\varepsilon_j}$. It follows that $\gamma_1^j>0$ which implies that $\phi_j$ is unstable.
\end{proof}

We also have that
\begin{theorem}
The equilibria $\phi_j$ are unstable, for $j\geq 3$ \emph{odd}. 
\end{theorem}	
\begin{proof}
	The idea here is to analyze the problem $D_j: H^2(0,\frac{\pi}{2})\cap H^1_0(0,\frac{\pi}{2})\to L^2(0,\frac{\pi}{2})$ defined by
$$
D_ju=u''+\frac{\lambda f'(\phi_j)}{a(\|\phi_j\|^2)}u.
$$
Now, observe that if $\gamma_2^j$ is the second eigenvalue of $L^{\phi_j}_{0}$ with associated eigenfunction $u_2$ then $\gamma_2^j$ is the first eigenvalue of $D_j$ and its eigenfunction is given by $u_2\big|_{[0,\frac{\pi}{2}]}.$
	
Now we will use the same idea as before to assure that $\gamma_2^j$ is positive. Remember that $\phi_j(x) >0$ in $(0,\tfrac{\pi}{j})$ and we have that 
$$
\phi_j(\tfrac{2\pi}{j} - x)=-\phi_j(x), \ x\in [0,\tfrac{2\pi}{j}], \quad \phi_j(0)=\phi_j(\pi)=0, \quad (\phi_j)'(0)>0.
$$	
Also, observe that $(\phi_j)'(\frac{3\pi}{2j})=0$ and  $(\phi_j)''(\frac{3\pi}{2j})=-\lambda f(\phi_j(\frac{3\pi}{2j}))>0$.

Consider the unique function $w\in H^2(0,\pi)$ that satisfies  
$$
\left\{\begin{aligned}
&w'' +\frac{\lambda f'(\phi_j)}{a(\|\phi_j'\|^2)}w =0,\\
&w(0)=0 \ w'(0)=1.
\end{aligned}\right.
$$
	
Observe that $w$ and $\phi'$ are both solutions of the same equation and then the Wronskian determined by the solutions must be constant
$$
(\phi_j)'(x)w'(x)-(\phi_j)''(x)w(x) =(\phi_j)'(0)>0.
$$	
Taking $x=\frac{3\pi}{2j},$ we have $(\phi_j)'(x)w'(x)-(\phi_j)''(x)w(x)=-(\phi_j)''(\frac{3\pi}{2j})w(\frac{3\pi}{2j})>0.$

Then,  $w(\frac{3\pi}{2j})<0$ and $\frac{3\pi}{2j}\leq \frac{\pi}{2}.$ Then, we have there exists at least one $x^* <\frac{\pi}{2}$ in such way that $w(x^*)=0,$ thus, $\gamma_2^j >0.$
	
The instability of $\phi_j$ for now follows by Lemma \ref{theo:eigenvalues}. 
\end{proof}

The results proved in this section ensure that, for any sign changing equilibria $\phi$, there exists a $\delta_0>0$ and, for each $\delta<\delta_0$, a $0<\delta'<\delta$ and a global solution of \eqref{eq_nl_changed} $\eta:\R \to H^1_0(0,\pi)$ such that
\begin{equation*}
\begin{split}
&u_0\in H^1_0(0,\pi),\  \|u_0-\phi\|_{H^1_0(0,\pi)}<\delta',\\
&\eta(0)=u_0\ \hbox{and}\ \|\eta(\tau)-\phi\|_{H^1_0(0,\pi)}\leq \delta, \hbox{ for all } \tau\leq 0, \\
&\|\eta(\tau)-\phi\|_{H^1_0(0,\pi)}\leq K e^{\beta \tau}\|u_0 - \phi\|_{H^1_0(0,\pi)}, \hbox{ for all } \tau \leq 0.
\end{split}
\end{equation*}
as before, making the change in the time-variable $t=\int_0^\tau a(\|\eta_x(\cdot,\theta)\|^2)^{-1}d\theta$, $\xi(t)=\eta(\tau)$ is a global solution of \eqref{eq_nl_changed} and, in this case, $\tfrac{\tau}{m} \leq t \leq \tfrac{\tau}{M}$. Therefore, for all $t \leq 0$,
$$
\|\xi(t)-\phi\|_{H^1_0(0,\pi)}\leq \|\eta(\tau)-\phi\|_{H^1_0(0,\pi)}\leq Ke^{{\beta}\tau}\|u_0-\phi\|_{H^1_0(0,\pi)}\leq Ke^{{\beta} mt}\|u_0-\phi\|_{H^1_0(0,\pi)}.
$$

\section{Hyperbolicity of the sign changing equilibria}

Recall that, as seen in the beginning of Section \ref{S2}, the hyperbolicity of the equilibria will follow if we prove that $0\notin \sigma(L_\varepsilon^\phi)$, for all equilibrium $\phi$ of \eqref{eq_non-local}, where $L_\varepsilon^\phi$ is given by \eqref{eq:linearization}. We already know, from Lemma \ref{Lemma-Henry}, that $0\notin \sigma(L_0^\phi)$.

\subsection{Hyperbolicity of $\phi_2$}

As we have seen in Section \ref{S2}, the linearization around $\phi_2$ of $w_t=w_{xx}+\frac{\lambda f(w)}{a(\|w_x\|^2)}$ is given by 
$$
L^{\phi_2}_{\varepsilon_2}u=u''+\frac{\lambda f'(\phi_2)u}{a(\|\phi_2'\|^2)}-\frac{2\lambda^2 a'(\|\phi_2'\|^2)}{a(\|\phi_2'\|^2)^3}f(\phi_2)\int_0^\pi f(\phi_2(s))u(s)ds.
$$

We want to prove that $0$ is not an eigenvalue of $L^{\phi_2}_{\varepsilon_2}$. By contradiction, 
assume that $0$ is an eigenvalue of $L^{\phi_2}_{\varepsilon_2}$. Hence, there exists a $0\neq v \in H^2(0,\pi)\cap H^1_0(0,\pi)$ such that $L^{\phi_2}_{\varepsilon_2}v=0$. By Theorem \ref{DD}, $0$ is a simple eigenvalue of $L^{\phi_2}_{\varepsilon_2}$ and, consequently, either $v(\pi-x)=v(x)$, for all $x \in [0,\pi]$, or $v(\pi -x)=-v(x)$, for all $x \in [0,\pi]$. Since $\int_0^\pi f(\phi_2)v\neq 0$, the second alternative holds. In particular, $v(\frac{\pi}{2})=0$. Using the symmetry properties verified above,
$$
\int_0^\pi f(\phi_2(s))v(s)ds=2\int_0^\frac{\pi}{2} f(\phi_2(s))v(s) ds.
$$

For $\varepsilon \in \mathbb{R}$, consider the operator $M_{\varepsilon}: D(M_{\varepsilon})\subset H^1_0(0,\frac{\pi}{2})\to H^1_0(0,\frac{\pi}{2})$ defined by $D(M_{\varepsilon})=H^2(0,\frac{\pi}{2})\cap H^1_0(0,\frac{\pi}{2})$ and
$$
M_{\varepsilon} u=u''+\frac{\lambda f'(\phi_2)u}{a(\|\phi_2'\|^2)}+\varepsilon f(\phi_2)\int_0^\frac{\pi}{2} f(\phi_2(s))u(s)ds, \ \forall u\in D(M_\epsilon).
$$

Now, for $u=v|_{[0,\frac{\pi}{2}]}$ we have $u\in D(M_{\varepsilon_2})$ and $M_{\varepsilon_2}u=0$, where $\varepsilon_2=-\frac{\lambda^2 4a'(\|\phi_2'\|^2)}{a(\|\phi_2'\|^2)^3}$ and then  $0 \in \sigma(M_{\varepsilon_2})$. By Theorem \ref{DD} and the fact that $\varepsilon_2\leq 0$, this would imply that there exists $\gamma \geq 0$ such that $\gamma$ is an eigenvalue for $M_0.$ 

But, this leads to a contradiction, since $M_0$ corresponds to the linearization of the semilinear problem
\begin{equation}\label{eq_sl_half}
\left\{
\begin{aligned}
& u_t = u_{xx} + \frac{\lambda f(u)}{a(\| \phi_2'\|^2)} ,  \, x\in (0, \tfrac{\pi}{2}),\, t>0,\\
& u(0, t)=u(\tfrac{\pi}{2}, t)=0,  \ \  t\geq 0,\\
& u(\cdot,0)=u_0(\cdot)\in H^1_0(0,\tfrac{\pi}{2}).
\end{aligned}
\right.
\end{equation}
around its positive equilibrium $\psi_1=\phi_2\big|_{[0,\pi]}$.

This concludes the proof that $\phi_2$ is hyperbolic.

\bigskip

\subsection{Hyperbolicity of $\phi_j$ for $j$ odd.}

Again, we wish to prove that $0\notin \sigma(L_{\varepsilon_j}^{\phi_j})$, where $L_{\varepsilon_j}^{\phi_j}$ is given by \eqref{eq:linearization}, with $\phi$ replaced by $\phi_j$ and $\varepsilon_j=-\frac{2\lambda^2 a'(\|\phi_j'\|^2)}{a(\|\phi_j'\|^2)^3}$. By contradiction, assume that we can find $0\neq u \in H^1_0(0,\pi)\cap H^2(0,\pi)$ such that $L^{\phi_j}_{\varepsilon_j}u=0$. The simplicity of the zero eigenvalue (see Theorem \ref{DD}), implies that $u(x)=u(\pi-x)$, for all $x \in [0,\pi].$

\bigskip 

There are two possible cases: Either $u(\frac{\pi}{j})=0$ or $u(\frac{\pi}{j})\neq 0$.

\bigskip

\textbf{Case $u(\frac{\pi}{j})=0$:}
In this case, $u$ will have the same symmetries that $\phi_j$ has. In fact, define 
\[u_1(x)=\begin{cases} u(\frac{\pi}{j}-x), & \ x\in [0,\frac{\pi}{j}] \\
-u(x-\frac{\pi}{j}), & \ x \in [\frac{\pi}{j},\pi]
\end{cases}\]

Observe that $L^{\phi_j}_{\varepsilon_j}u_1=0$ and, using the simplicity, we have that $u_1=\pm u$. Let us show that $u_1=-u$ is not possible. If that was the case $u(x)=-u(\frac{\pi}{j}-x)$ for $x\in [0,\frac{\pi}{j}]$ and $u(x)=u(\frac{\pi}{j}+x)$, $x\in [0, \pi-\frac{\pi}{j}]$. This would lead to
$$
\int_0^\pi f(\phi_j(s))u(s)ds=0
$$
and that $0$ is an eigenvalue of $L_0^\phi$ which is a contradiction. 

Hence $u=u_1$ and $u(x)=u(\frac{\pi}{j}-x)$ for all $x \in [0,\frac{\pi}{j}]$ and $u(x)=-u(x-\frac{\pi}{j})$, $x\in [\frac{\pi}{j},\pi]$.

Therefore, $u$ has the same symmetry as $\phi_j$. Because of that, we can guarantee $v=u|_{[0,\frac{\pi}{j}]}$ satisfies 
$$
v''+\frac{\lambda f'(\phi_j)}{a(\|\phi_j'\|^2)}v+j\varepsilon_j f(\phi_j)\int_0^{\frac{\pi}{j}} f(\phi_j(s))v(s)ds=0 
$$
and we can apply the same reasoning employed in the case $j=2$ to arrive at a contradiction.

\bigskip

\textbf{Case $u(\frac{\pi}{j})\neq 0$:} We define the following auxiliary functions
\[u_1(x)=\begin{cases} u(x+\tfrac{\pi}{j}), & \ x\in [0,\frac{(j-1)\pi}{j}] \\
-u(x-\tfrac{(j-1)\pi}{j}), & \ x \in [\tfrac{(j-1)\pi}{j},\pi]
\end{cases}\]
and
$u_2(x)=u_1(\pi-x)$ for all $x \in [0,\pi]$.

We know that $\int_0^{\pi} f(\phi_j(s))u(s)ds=-\int_0^{\pi} f(\phi_j(s))u_1(s)=-\int_0^{\pi} f(\phi_j)u_2(s)ds$. We observe that, despite $u_1, u_2 \notin H^1_0(0,\pi)$ we have  $L^{\phi_j}_{\varepsilon_j}u_1=L^{\phi_j}_{\varepsilon_j}u_2=0$ (here we use the extension of $L^{\phi_j}_{\varepsilon_j}$ to $H^2(0,\pi)$).

Next we use the construction done in \cite{Catchpole} to obtain solutions of the local operator $L_0^{\phi_j}$ using linearly independent solutions of the operator $L_{\varepsilon_j}^{\phi_j}$.

Then, $v_1(x)=u(x)+u_1(x)$ and $v_2(x)=u(x)+u_2(x)$ satisfy,
\begin{equation}\label{eq-aux2}
v''+\frac{\lambda f'(\phi_j)}{a(\|\phi_j'\|^2)}v=0.
\end{equation}
Let us prove that $\{v_1,v_2\}$ defines a fundamental set of solutions for \eqref{eq-aux2}. In fact, if $v_1=\alpha v_2$, then $0\neq v_1(0)=u(\frac{\pi}{j})=\alpha v_2(0)=-\alpha u(\frac{\pi}{j})$. Hence $\alpha=-1$ and $2u+u_1+u_2\equiv 0$. Computing this identity at $k\frac{\pi}{j}$ for $k=1,2,\cdots,j-1$ we obtain the formula
\begin{equation}\label{eq-aux3}
u\left(\frac{(k-1)\pi}{j}\right)+2u\left(\frac{k\pi}{j}\right)+u\left(\frac{(k+1)\pi}{j}\right), \ k=1,2,\cdots,j-1.
\end{equation}
which leads to  $\mathcal{L}\,\mathcal{U}=0$, where   
$$
\mathcal{L} =
\begin{bmatrix}
2               & 1        & 0     & \hdots   &\hdots&0   	& 0     & 0      \\
1 		     & 2     & 1       & \hdots   &\hdots&0     & 0     & 0         \\
0		         & 1     & 2     &\hdots   &\hdots&0     & 0     & 0        \\
 \vdots        & \vdots& \vdots&\ddots   &\ddots&\vdots& \vdots& \vdots\\
 \vdots        & \vdots& \vdots& \ddots   &\ddots&\vdots& \vdots& \vdots\\
0     & 0     & 0       &\hdots   &\hdots& 2     &1        & 0    \\
0     & 0     & 0        &\hdots   &\hdots& 1     & 2     &1       \\
0       & 0     & 0       &\hdots   &\hdots& 0     & 1     &2       \\
\end{bmatrix}
\quad \hbox{and}\quad \mathcal{U}=
\begin{bmatrix}
u(\tfrac{\pi}{j}) \\\vspace{-1pt}
u(\tfrac{2\pi}{j})\\\vspace{-1pt}
u(\tfrac{3\pi}{j}) \\ \vspace{-4pt}
\vdots \\ \vspace{-1pt}
\vdots\\\vspace{-1pt}
u(\tfrac{(j-3)\pi}{j}) \\\vspace{-1pt}
u(\tfrac{(j-2)\pi}{j})\\\vspace{-1pt}
u(\tfrac{(j-1)\pi}{j})
\end{bmatrix},
$$
consequently, since ${\rm det} (\mathcal{L})=j\neq 0$, we have that $u(k\frac{\pi}{j})=0$, $k=1,2,\cdots,j-1$, contradicting the assumption that $u(\frac{\pi}{j})\neq 0$. This concludes the proof that $\{v_1,v_2\}$ defines a fundamental set of solutions for \eqref{eq-aux2}.


Since $(\phi_j)'$ also satisfies \eqref{eq-aux2}, there are real numbers $\alpha$ and $\beta$ such that $(\phi_j)'=\alpha v_1+\beta v_2$.

Recall that,
\begin{equation}\label{eq:sign.deriv}
(\phi_j)'(0)=(-1)^k(\phi_j)'(\tfrac{k\pi}{j}) \mbox{ for all } k \in \{1, \dots, j\}.
\end{equation}

Observe that, for all $x \in [0,\pi]$, 
$$
\begin{aligned}
&(\phi_j)'(\pi-x)=-(\phi_j)'(x)\ \hbox { and } \\
&v_1(\pi-x)=u(\pi-x)+u_1(\pi-x)=u(x)+u_2(x)=v_2(x).
\end{aligned}
$$
Therefore,
$\alpha v_1(x)+\beta v_2(x)=
-\alpha v_2(x)-\beta v_1(x)$, that is, $(\alpha+\beta) (v_1(x)+v_2(x))=0$, for all $x \in [0,\pi]$. Since $v_1+v_2\neq 0$, we conclude that $\beta=-\alpha$. Hence,
$$
(\phi_j)'=\alpha v_1+\beta v_2=\alpha (u_1-u_2).
$$

Now,\vspace{-5pt}
\begin{equation}\label{eq:derivadas}
\begin{aligned}
(\phi_j)'(0)&=
2\alpha u(\tfrac{\pi}{j})\\
(\phi_j)'(\tfrac{k\pi}{j})&=
\alpha u(\tfrac{(k+1)\pi}{j})-\alpha u(\tfrac{(k-1)\pi}{j}) \quad (1\leq k\leq j-1)
\end{aligned}
\end{equation}

Then, from \eqref{eq:sign.deriv} and \eqref{eq:derivadas}, noting that
$\alpha\neq 0$ (otherwise, $(\phi_j)'=0$), we have
\begin{equation}\label{eq:relation}
2(-1)^ku\left(\frac{\pi}{j}\right)-u\left(\frac{(k+1)\pi}{j}\right)+u\left(\frac{(k-1)\pi}{j}\right)=0, \ k=1,\cdots,j-1.
\end{equation}

\medskip

Next we consider, separately, the cases: $j=3$, $j=5$ and $j\geq 7$.

\medskip

For $j\!=\!3$, we can apply \eqref{eq:relation} for $k=1$ and, since $u(\tfrac{\pi}{3})=u(\tfrac{2\pi}{3})$, we have that $u(\tfrac{\pi}{3})=0$, which is a contraction. For $j=5$, using \eqref{eq:relation}, with $k=1,2$, we obtain
\begin{equation*}\begin{aligned}
&-2u(\tfrac{\pi}{5})-u(\tfrac{2\pi}{5})=0\\
&2u(\tfrac{\pi}{5})-u(\tfrac{3\pi}{5})+u(\tfrac{\pi}{5})=0.
\end{aligned}
\end{equation*}
Now, since  $u(\frac{3\pi}{5})=u(\frac{2\pi}{5})$, the above equations can be written as
$$ 
\begin{bmatrix}
-2 & -1\\
3 & -1
\end{bmatrix}
\begin{bmatrix}
u(\tfrac{\pi}{5})\\
u(\tfrac{2\pi}{5})
\end{bmatrix}
=\left[
\begin{array}{c}
0\\ 0
\end{array}
\right]
$$

Therefore, $u\left(\tfrac{\pi}{5}\right)=u\left(\!\tfrac{2\pi}{5}\!\right)=0$. This contradicts our assumption that $u(\tfrac{\pi}{5})\neq 0$.

For values of $j=2n+1$, $n\geq 3$, using \eqref{eq:relation} and the fact that $u(x)=u(\pi-x)$, for all $x\in [0,\pi]$, we obtain
\begin{equation}\label{eq:gen.matrix}
\left[
\begin{matrix}
-2               \hspace{-12pt} &-1        &\ \ 0     &\ \ 0     &\hdots   &\hdots&0   	&\ \ 0     &\ \ 0     &\ \ 0    \quad \\
\ \ 3 		    \hspace{-12pt} &\ \ 0     & -1       &\ \ 0     &\hdots   &\hdots&0     &\ \ 0     &\ \ 0     &\ \ 0    \quad \\
-2 		        \hspace{-12pt} &\ \ 1     &\ \ 0     &-1        &\hdots   &\hdots&0     &\ \ 0     &\ \ 0     &\ \ 0    \quad \\
\ \ 2 		    \hspace{-12pt} &\ \ 0     &\ \ 1     &\ \ 0     &\hdots   &\hdots&0     &\ \ 0     &\ \ 0     &\ \ 0    \quad \\
\ \ \vdots       \hspace{-12pt} &\ \ \vdots&\ \ \vdots&\ \ \vdots&\ddots   &\ddots&\vdots&\ \ \vdots&\ \ \vdots&\ \ \vdots\quad\\
\ \ \vdots       \hspace{-12pt} &\ \ \vdots&\ \ \vdots&\ \ \vdots&\ddots   &\ddots&\vdots&\ \ \vdots&\ \ \vdots&\ \ \vdots\quad\\
2(\!-1)^{n-3}    \hspace{-12pt} &\ \ 0     &\ \ 0     &\ \ 0     &\hdots   &\hdots&0     &-1        &\ \ 0     &\ \ 0    \quad \\
2(\!-1)^{n-2}    \hspace{-12pt} &\ \ 0     &\ \ 0     &\ \ 0     &\hdots   &\hdots&1     &\ \ 0     &-1        &\ \ 0    \quad \\
2(\!-1)^{n-1}    \hspace{-12pt} &\ \ 0     &\ \ 0     &\ \ 0     &\hdots   &\hdots&0     &\ \ 1     &\ \ 0     &-1       \quad \\
2(\!-1)^{n}      \hspace{-12pt} &\ \ 0     &\ \ 0     &\ \ 0     &\hdots   &\hdots&0     &\ \ 0     &\ \ 1     &-1       \quad \\
\end{matrix}\!\!\!\!\!
\right]\!\!\!
\begin{bmatrix}
u(\tfrac{\pi}{j}) \\
u(\tfrac{2\pi}{j})\\
u(\tfrac{3\pi}{j}) \\
u(\tfrac{4\pi}{j}) \\
\vdots \\
\vdots\\
u(\tfrac{(n-2)\pi}{j}) \\
u(\tfrac{(n-1)\pi}{j})\\
u(\tfrac{n\pi}{j})
\end{bmatrix}
\!\!=\!\!
\left[\begin{array}{c}
0 \\
0\\
0 \\
0 \\
\vdots \\
\vdots\\
0 \\
0\\
0\\
0
\end{array}\right]
\end{equation}

This $n\times n$ matrix defines a relation between the values of $u(\tfrac{k\pi}{j})$ for $k=1,\cdots, n$. Next we prove that this matrix is non-singular.

\begin{lemma}\label{lemma:determinante}
	Consider the matrix $A_2=\left[\begin{smallmatrix}
	-2 & -1\\
	3 & -1
	\end{smallmatrix}\right]$ and, for $n \in \mathbb{N}$, define $A_n$ as the $n\times n$ matrix give in \eqref{eq:gen.matrix}.  Then, for all $n\geq 2$, ${\rm det}(A_n)=(2n+1)(-1)^n$. 
\end{lemma}
\begin{proof} The proof is done by induction. Note that, ${\rm det}(A_2)=5$. Assume that ${\rm det}(A_n)=(-1)^n(2n+1)$ and let us prove that ${\rm det}(A_{n+1})=(-1)^{n+1}(2n+3)$.
	
	Consider the auxiliary matrix
$$
B_{n+1}=\left[
\begin{smallmatrix} \\
-2               \hspace{-6pt} &-1        &\ \ 0     &\ \ 0     &\hdots   &\hdots&0   	&\ \ 0     &\ \ 0     &\ \ 0    \ \\
\ \ 3 		    \hspace{-6pt} &\ \ 0     & -1       &\ \ 0     &\hdots   &\hdots&0     &\ \ 0     &\ \ 0     &\ \ 0    \ \\
-2 		        \hspace{-6pt} &\ \ 1     &\ \ 0     &-1        &\hdots   &\hdots&0     &\ \ 0     &\ \ 0     &\ \ 0    \ \\
\ \ 2 		    \hspace{-6pt} &\ \ 0     &\ \ 1     &\ \ 0     &\hdots   &\hdots&0     &\ \ 0     &\ \ 0     &\ \ 0    \ \\
\ \ \vdots       \hspace{-6pt} &\ \ \vdots&\ \ \vdots&\ \ \vdots&\ddots   &\ddots&\vdots&\ \ \vdots&\ \ \vdots&\ \ \vdots\ \\
\ \ \vdots       \hspace{-6pt} &\ \ \vdots&\ \ \vdots&\ \ \vdots&\ddots   &\ddots&\vdots&\ \ \vdots&\ \ \vdots&\ \ \vdots\ \\
2(\!-1)^{n-3}    \hspace{-6pt} &\ \ 0     &\ \ 0     &\ \ 0     &\hdots   &\hdots&0     &-1        &\ \ 0     &\ \ 0    \ \\
2(\!-1)^{n-2}    \hspace{-6pt} &\ \ 0     &\ \ 0     &\ \ 0     &\hdots   &\hdots&1     &\ \ 0     &-1        &\ \ 0    \ \\
2(\!-1)^{n-1}    \hspace{-6pt} &\ \ 0     &\ \ 0     &\ \ 0     &\hdots   &\hdots&0     &\ \ 1     &\ \ -1     &-1       \ \\
2(\!-1)^{n}      \hspace{-6pt} &\ \ 0     &\ \ 0     &\ \ 0     &\hdots   &\hdots&0     &\ \ 0     &\ \ 0     &-1       \ \\
\end{smallmatrix}\right].
$$
	
Note that, ${\rm det}(A_{n+1})={\rm det}(B_{n+1})$ since $B_{n+1}$ can be obtained from $A_{n+1}$ (by adding the last column to the second to the last column of $A_{n+1}$). Now, $B_{n+1}$ can be written as
$$
B_{n+1}=\left[\begin{array}{c|r}\scriptstyle &\scriptstyle 0\\
\scriptstyle& \scriptstyle 0 \\
A_n &\scriptstyle \vdots\\
&\scriptstyle 0 \\
&\scriptstyle -1\\ \hline
\scriptstyle 2(-1)^{n+1} \ 0 \ 0 \ 0 \ \hdots \ 0 \ 0  \ 0 &\scriptstyle -1
\end{array}
\right]
\qquad
B_{n+1}=\left[\begin{array}{l|c}
\scriptstyle -2&   \\
\scriptstyle3&  \\
\scriptstyle-2& T_n  \\ 
\scriptstyle \vdots & \\
\scriptstyle \!2(\!-1\!)^{n}\! &\\
\hline
\scriptstyle \!2(\!-1\!)^{n+1}\! &\scriptstyle 0\ 0 \ 0 \  \dots \ 0 \ 0\  -1
\end{array}
\right]
$$
where $T_n$ is a triangular matrix of order $n$ with $-1$ in the diagonals entries.

Hence, using the Laplace expansion related to the last line, we arrive at
$$
{\rm det}(B_{n+1}) =-{\rm det}(A_n) -2{\rm det}(T_n).
$$

Now, using the induction hypothesis and making some calculations, we conclude that ${\rm det}(A_{n+1}) ={\rm det}(B_{n+1}) =(-1)^{n+1}[2(n+1)+1]$
as desired.
\end{proof}

From Lemma \ref{lemma:determinante}, $A_n$ is non-singular and, consequently, $u(\tfrac{\pi}{j})=0$, which is a contradiction. We conclude that there does not exist a non-zero function $u$ satisfying $L^{\phi_j}_{\varepsilon_j}u=0$. Hence $0 \notin \sigma(L^{\phi_j}_{\varepsilon_j})$ for all $j=2n+1$ for $n \in \mathbb{N}$. Therefore, $\phi_j$ is hyperbolic.

\subsection{Hyperbolicity of $\phi_j$ for $j$ even} Suppose that, there exists $0\neq u\in H^1(0,\pi)\cap H^1_0(0,\pi)$ such that $L^{\phi_j}_{\varepsilon_j} u=0$.

Since $j$ is even, there are non-negative integers $n\geq 1$ and $k\geq 0$ such that $j=2^n(2k+1).$ 

\begin{lemma}\label{lemma:sym-u}
If $u$ is an eigenfunction associated to the eigenvalue $0$ of $L^{\phi_j}_{\varepsilon_j}$, the following holds:
$$
u(\tfrac{\pi}{2^i}-x)=-u(x),\  x \in \left[0,\tfrac{\pi}{2^i}\right], \quad 1\leq i < n.
$$	
\end{lemma}
\begin{proof} 
	This result is a consequence of the symmetries of $\phi_j$. For $j$ even, with $j=2^n(2k+1)$, we have, for $1\leq i < n$, \[\phi_j(\tfrac{\pi}{2^i}-x)=-\phi_j(x),\quad  x \in [0,\tfrac{\pi}{2^i}]. \]
	
We know, from Theorem \ref{DD}, that $0$ is a simple eigenvalue of $L_{\varepsilon_j}^{\phi_j}$. Thus, either $u(\pi-x)= u(x)$, for $x \in [0,\pi]$ or $u(\pi-x)=- u(x)$, for $x \in [0,\pi]$. Since $\int_0^\pi f(\phi_j(s))u(s)ds \neq 0$, for $0$ is not an eigenvalue of $L_0^{\phi_j}$, we have that $u(\pi-x)=-u(x)$, $x \in [0,\pi]$ and $u(\frac{\pi}{2})=0$.

Now,\[\begin{split}
\int_0^\pi f(\phi_j(s))u(s)ds&=\int_0^\frac{\pi}{2} f(\phi_j(s))u(s)ds+\int_\frac{\pi}{2}^\pi f(\phi_j(s))u(s)ds\\
&=\int_0^\frac{\pi}{2} f(\phi_j(s))u(s)ds+\int_0^\frac{\pi}{2}f(\phi_j(\pi-x))u(\pi-x)dx\\
&=2\int_0^\frac{\pi}{2} f(\phi_j(s))u(s)ds.
\end{split}\]

Define the auxiliary function
$$
{u}_1(x)=\begin{cases}
u(\frac{\pi}{2}-x),&  x \in [0,\frac{\pi}{2}]\\
-u(x-\frac{\pi}{2}),&  x \in [\frac{\pi}{2},\pi].
\end{cases}
$$
Since $\phi_j$ is antisymmetric in $[0,\frac{\pi}{2}]$, ${u}_1$ is an eigenfunction of $L^{\phi_j}_{\varepsilon_j}$ associated to the zero eigenvalue, we have $u=\pm u_1$. In particular, we have either $u(\frac{\pi}{2}-x)= u(x)$, for $x \in [0,\frac{\pi}{2}]$, or $u(\frac{\pi}{2}-x)=- u(x)$, for $x \in [0,\frac{\pi}{2}]$. Since $\int_0^\frac{\pi}{2} f(\phi_j(s))u(s)ds\neq 0,$ the second alternative occurs. 

By induction, using auxiliary functions 
$$
{u}_i(x)=\begin{cases}
u(\frac{\pi}{2^i}-x),&  x \in [0,\frac{\pi}{2^i}],\\
-u(x-\frac{\pi}{2^i}),&  x \in [\frac{\pi}{2^i},\pi],
\end{cases}
$$
the result follows for $i=1,2,\cdots,n-1$.
\end{proof}

We can use the symmetries provided by Lemma \ref{lemma:sym-u}, in order to conclude that 
$$
\int_0^\pi f(\phi_j(s))u(s)ds={2^n} \int_0^\frac{\pi}{2^n} f(\phi_j(s))u(s)ds.
$$

Defining $\psi_j\!=\!\phi_j|_{[0,\tfrac{\pi}{2^n}]}$ and $v\!=\!u|_{[0,\tfrac{\pi}{2^n}]}$ and noting that $v \!\in\! H^2(0,\frac{\pi}{2^n})\cap H^1_0(0,\frac{\pi}{2^n})$ and satisfies 
\begin{equation}\label{eq:subint}
v''+\frac{\lambda f'(\psi_j)}{a(2^n\|\psi_j'\|^2)}v-2^{n+1}\frac{\lambda^2 a'(2^n\|\psi_j'\|^2)}{a(2^n\|\psi_j'\|^2)^3}f(\psi_j)\int_0^\frac{\pi}{2^n} f(\psi_j(s))v(s)ds=0.
\end{equation} 

Observe that $\psi_j$ is an equilibrium of
\begin{equation}\label{eq_aux4}
\left\{\begin{aligned}
& u_t =d(\|u_x\|_n^2)u_{xx}+\lambda f(u),\ x\in (0,\tfrac{\pi}{2^n}),\ t>0,\\
&u(0,t)=u(\tfrac{\pi}{2^n},t)=0,\ t\geq 0,\\
& u(\cdot,0)=u_0(\cdot) \in H^1_0(0,\tfrac{\pi}{2^n}),
\end{aligned}\right.
\end{equation}
where $d(\cdot)=a(2^n \cdot)$ and $\|u_x\|_n^2=\int_0^\frac{\pi}{2^n} |u_x(s)|^2ds$. 

If $\tilde{\phi}_{2k+1}$ represents the equilibrium of \eqref{eq_aux4} that has $2(k+1)$ zeros in $[0,\tfrac{\pi}{2^n}]$ and is such that $\tilde{\phi}_{2k+1}(x)>0$, $x\in [0,\tfrac{\pi}{k2^n}]$, then we have $\psi_j=\tilde{\phi}_{2k+1}$.

Now, \eqref{eq:subint} implies that $v$ satisfies
$$
L^{\psi_j}_{\varepsilon_j}v=L^{\psi_j}_{0}v-\tfrac{2\lambda^2 d'(\|(\psi_j)'\|_n^2)}{d(\|(\psi_j)'\|_n^2)^3}f(\psi_j)\int_{0}^{\frac{\pi}{2^n}} f(\psi_j(s))v(s)ds=0.
$$
Consequently, $0$ is an eigenvalue of $L^{\psi_j}_{\varepsilon_j}$ and $v$ is a corresponding eigenvector.

Now, we can apply the reasoning of the case $j$ odd to arrive at a contradiction. This completes the proof that $\phi_j$ is hyperbolic for $j$ even.

\section{Conclusion}

In this article we study the problem \eqref{eq_non-local}, where $a$ is globally Lipschitz non-decreasing function and $f$ has the profile of a cubic function of the form $u-u^3$. In that case, the semigroup (solution operator) associated to \eqref{eq_non-local} has a global attractor. Also, \eqref{eq_non-local} has a Lyapunov function, given by \eqref{LyapunovFunction}, and, for $a(0)N^2< \lambda\leq a(0)(N+1)^2$, \eqref{eq_non-local} has $2N+1$ equilibria (Theorem \ref{theo:exist.equilibria}), this and \cite[Theorem 3.8.6]{Hale} implies that the global attractor has the structure seen in \eqref{Att_charac}. This also proves that each equilibrium is topologically hyperbolic in the sense of Definition \ref{TH}, from the results in \cite[Lemma 2.18]{BCLBook}.

We prove that these equilibria are also strictly hyperbolic in the sense of Definition \ref{SH}. Taking advantage of the `change of variables', that transforms the quasilinear problem \eqref{eq_non-local} into the semilinear problem \eqref{eq_nl_changed}, we apply semilinear techniques to ensure strict hyperbolicity for the equilibria of \eqref{eq_non-local}, through a spectral analysis of a nonlocal linear operator. The spectral analysis of this operator is very interesting and challenging. It allows us to deduce that the local stable and unstable manifolds of the equilibria of \eqref{eq_nl_changed} are given in the form expressed in Definition \ref{SH} (see below). Since the change of variables (one for each solution) does not affect the state variable, the same holds for \eqref{eq_non-local}.

\bigskip

For $a(0)N^2< \lambda< a(0)(N+1)^2$, given an equilibrium $\phi$ of \eqref{eq_nl_changed}, consider the operator $L_\varepsilon^\phi$ given by \eqref{Lepsphi}. We have proved that $0$ is not in the spectrum $\sigma(L_\varepsilon^\phi)$ of $L_\varepsilon^\phi$. Let $P_u$ be the spectral projection, associated with the part of $\sigma(L_\varepsilon^\phi)$ to the right of the imaginary axis, and $P_s$ be the spectral projection, associated with the part of $\sigma(L_\varepsilon^\phi)$ to the left of the imaginary axis and define $X_u=P_u(H^1_0(0,\pi))$ and $X_s=P_s(H^1_0(0,\pi))$. Now, \cite[Theorem 4.4]{BCLBook} ensures the existence of $\theta_s$ and $\theta_u$ with the properties required in Definition \ref{SH}. 

Hence, there is a $\delta_0>0$ and constants $K,\beta>0$ in such a way that, for each $0<\delta<\delta_0$ a $0<\delta'<\delta$, such that, if $\|x_u^0\|_{H^1_0(0,\pi)}<\delta'$, there is a global solution $\eta:\R\to H^1_0(0,\pi)$ of \eqref{eq_nl_changed} with $\eta(0)=(x_u^0,\theta_u(x_u^0))$, $\|\eta(\tau)-\phi\|_{H^1_0(0,\pi)}< r$, $\eta(\tau)=\phi+(P_u(\eta(\tau)),\theta_u(P_u(\eta(\tau))))$ for all $\tau\leq 0$ and, for some $\beta>0$,
$$
\|\eta(\tau)-\phi\|_{H^1_0(0,\pi)}\leq K e^{\beta \tau}\|\xi(0) - \phi\|_{H^1_0(0,\pi)}, \quad \forall \tau \leq 0.
$$
also if $\|x_s^0\|_{H^1_0(0,\pi)}<\delta'$, the solution $w:\R^+\to H^1_0(0,\pi)$ of \eqref{eq_nl_changed} such that $w(0)=(x_s^0,\theta_w(x_s^0))$ satisfies $\|w(\tau)-\phi\|_{H^1_0(0,\pi)}< \delta$, for all $\tau\geq 0$, $w(\tau)=\phi+(P_s(w(\tau)),\theta_u(P_s(w(\tau))))$, for all $\tau\geq 0$, and, for some $\beta>0$,
$$
\|w(\tau)-\phi\|_{H^1_0(0,\pi)}\leq K e^{-\beta \tau}\|w(0) - \phi\|_{H^1_0(0,\pi)}, \quad \forall \tau \geq 0.
$$

\bigskip

Changing the time-variable to $t=\int_0^\tau a(\|\eta_x(\cdot,\theta)\|^2)^{-1}d\theta$ and making
$\xi(t)=\eta(\tau)$, in the first case, or to  $t=\int_0^\tau a(\|w_x(\cdot,\theta)\|^2)^{-1}d\theta$ and making $u(t)=w(\tau)$, in the second case, we have also the result for \eqref{eq_non-local}
$$
\|\xi(t)-\phi\|_{H^1_0(0,\pi)}\leq K e^{\beta mt}\|\xi(0) - \phi\|_{H^1_0(0,\pi)}, \quad \forall t \leq 0,
$$
and also
$$
\|u(t)-\phi\|_{H^1_0(0,\pi)}\leq K e^{-\beta mt}\|u(0) - \phi\|_{H^1_0(0,\pi)}, \quad \forall t \geq 0.
$$

We emphasize that the strict hyperbolicity for quasilinear problems still lacks, in general, a method of proof. In the example treated in this paper, that can be accomplished through its relation with a semilinear problem.

\bibliographystyle{amsplain}

\begin{thebibliography}{10}


	\bibitem{BCLBook}
	\newblock M. C. Bortolan, A.N. Carvalho and J. A. Langa,
	\newblock \emph{Attractors under autonomous and non-autonomous perturbations},
	\newblock Mathematical Surveys and Monographs, AMS, 2020.







\bibitem{CLRBook} (MR2976449)
\newblock A.N. Carvalho, J. A. Langa and J. C. Robinson,
\newblock \emph{Attractors for infinite-dimensional non-autonomous dynamical systems},
\newblock Springer: New York, 2013.


\bibitem {CLLM} \newblock A. N. Carvalho and T. L. M. Luna and Y. Li and E. M. Moreira,
\newblock \emph{A non-autonomous bifurcation problem for a non-local scalar one-dimensional parabolic equation},
\newblock To appear. 

\bibitem{Catchpole} 
\newblock E. A. Catchpole,
\newblock {A {C}auchy problem for an ordinary integro-differential equation},
\newblock \emph{Proc. Roy. Soc. Edinburgh Sect. A} \textbf{72} (1) (1974), 39--55. MR397342
    \bibitem{CH-IN} 
\newblock N. Chafee and E. F. Infante,
\newblock {A bifurcation problem for a nonlinear partial differential equation of parabolic type},
\newblock \emph{Applicable Anal.} \textbf{4}  (1974), 17--37. MR440205

   \bibitem{Chaf-Inf}
\newblock N. Chafee and E. F. Infante,
\newblock  {Bifurcation and stability for a nonlinear parabolic partial differential equation}.
\newblock \emph{Bull. Amer. Math. Soc.} \textbf{80} (1974), 49--52. MR328359
		
\bibitem {CH-R} 
\newblock M. Chipot and J.-F. Rodrigues,
\newblock {On a class of nonlocal nonlinear elliptic problems},
\newblock \emph{RAIRO Mod\'{e}l. Math. Anal. Num\'{e}r.} \textbf{26} (3) (1992), 447--467. MR1160135


		
\bibitem {CH-VA-VC} 
\newblock M. Chipot, V. Valente and G. Vergara Caffarelli,
\newblock {Remarks on a nonlocal problem involving the {D}irichlet energy},
\newblock \emph{Rend. Sem. Mat. Univ. Padova} \textbf{110} (2003), 199--220. MR2033009


	\bibitem{DavidsonDodds} 
\newblock {F. A. Davidson and N. Dodds},
\newblock {Spectral properties of non-local differential operators},
\newblock \emph{Appl. Anal.} \textbf{85} (6-7) (2006), 717--734. MR2232418

	\bibitem{Dodds} 
\newblock N. Dodds;
\newblock {Further spectral properties of uniformly elliptic operators that include a non-local term}.
\newblock \emph{Appl. Math. Comput.} \textbf{197} (1) (2008), 317--327. MR2396314

\bibitem{Freitas} 
\newblock P. Freitas,
\newblock   {A nonlocal {S}turm-{L}iouville eigenvalue problem}.
\newblock \emph{Proc. Roy. Soc. Edinburgh Sect. A} {\bf 124} (1)  {(1994)}, {169--188}. MR1272438

\bibitem{HE} 
\newblock D. Henry,
\newblock \emph{Geometric theory of semilinear parabolic equations},
\newblock  Berlin: Springer, 1981. MR610244



\bibitem{Kri} 
\newblock G. A. Kriegsmann,
\newblock {Hot spot formation in microwave heated ceramic fibres},
\newblock \emph{IMA Journal of Applied Mathematics} {\bf 59} (1997), 123-148. MR1482417

\bibitem{Phillipo} 
\newblock P. Lappicy,
\newblock {Sturm attractors for quasilinear parabolic equations},
\newblock \emph{J. Differential Equations} {\bf 265} (2018), 4642--4660. (MR3843311)

\bibitem{Hale} 
\newblock J. K. Hale,
\newblock \emph{Asymptotic behavior of dissipative systems},
\newblock Mathematical Survey and Monographs, AMS, 1989. MR941371



%

\bibitem{Smoller} 
\newblock J. Smoller,
\newblock \emph{Shock waves and reaction-diffusion equations},
\newblock  New York: Springer-Verlag, 1994. MR1301779

\end{thebibliography}

\end{document}